\documentclass[smallextended]{svjour3}
\usepackage{amsfonts}
\usepackage{latexsym}
\usepackage{amsmath}
\usepackage{amssymb}
\usepackage{amscd}
\usepackage{epsfig}
\usepackage{multirow}
\usepackage{graphicx}
\addtocounter{MaxMatrixCols}{4}
\renewenvironment{proof}{\par {\sc {\bf Proof.}\hskip 5pt}}{\hfill \qed \par}
\newcommand{\undertilde}[1]{\ensuremath{\mathord{\vtop{\ialign{##\crcr
   $\hfil\displaystyle{#1}\hfil$\crcr\noalign{\kern1.5pt\nointerlineskip}
   $\hfil\tilde{}\hfil$\crcr\noalign{\kern1.5pt}}}}}}
\begin{document}

\title{Directed in-out graphs of optimal size}
\author{D.~Glynn, M.~Haythorpe, and A.~Moeini}

\institute{David Glynn
\at Flinders University\\
\email{david.glynn@flinders.edu.au} \and Michael Haythorpe (Corresponding author)
\at Flinders University, Ph: +61 8 8201 2834, Fax: +61 8 8201 2904\\
\email{michael.haythorpe@flinders.edu.au} \and Asghar Moeini
\at University of Melbourne\\
\email{asghar.moeini@unimelb.edu.au}} \maketitle
{\abstract
We discuss the recently introduced concept of {\em k-in-out graphs}, and provide a construction for $k$-in-out graphs for any positive integer $k$. We derive a lower bound for the number of vertices of a $k$-in-out graph for any positive integer $k$, and demonstrate that our construction meets this bound in all cases. For even $k$, we also prove our construction is optimal with respect to the number of edges, and results in a planar graph. Among the possible uses of in-out graphs, they can convert the generalized traveling salesman problem to the asymmetric traveling salesman problem, avoiding the \lq\lq big M" issue present in most other conversions. We give constraints satisfied by all in-out graphs to assist cutting-plane algorithms in solving instances of traveling salesman problem which contain in-out graphs.


}

\section{Introduction}\label{sec-Introduction}
Consider a simple, connected, directed graph $G$ of order $n$. The {\em Hamiltonian cycle problem} (HCP) is: determine if there exists at least one simple cycle of length $n$ in the graph. Such simple cycles of length $n$ are called {\em Hamiltonian cycles} (HC) and graphs containing at least one HC are called {\em Hamiltonian}. Similarly, a simple path of length $n$ is called a {\em Hamiltonian path}.

We consider a family of graphs which possess a property called {\em in-out}, recently defined in Haythorpe and Johnson \cite{bellringing}. Interchangably, they may be referred to as {\em in-out graphs} or {\em in-out subgraphs}, with the latter name used because many applications occur when they are included as part of a larger graph. In-out graphs are defined as follows.

\begin{definition} Consider a graph $S$ and suppose that, for some positive integer $k$, there are $k$ vertices in $S$ that are labelled $i_1, \hdots, i_k$ and called the $k$ {\em incoming vertices}, while $k$ vertices (possibly overlapping with the set of $k$ incoming vertices) in $S$ are labelled $o_1, \hdots, o_k$ and called the $k$ {\em outgoing vertices}. Then $S$ is called a {\em $k$-in-out graph} if it satisfies the following two conditions.

\begin{enumerate}\item For all $j,m = 1, \hdots, k$, there is a Hamiltonian path in $S$ between vertices $i_j$ and $o_m$ if and only if $j = m$.
\item There is no union of more than one disjoint paths in $S$, each starting at an incoming vertex and finishing at an outgoing vertex, such that all vertices in $S$ are visited.\end{enumerate}

We refer to the first condition as the {\em paired vertices condition}, and the second condition as the {\em single visit condition}.\label{def-ios}\end{definition}

As mentioned previously, many applications occur when an in-out graph $S$ is included as part of a larger graph. A primary such application occurs in the context of HCP. Consider any graph $G$ satisfying the following conditions.

\begin{enumerate}\item $S$ is an induced subgraph of $G$.
\item Any edges going from $G \setminus S$ to $S$ (which we call {\em incoming edges}) are incident with one of the incoming vertices of $S$.
\item Any edges going from $S$ to $G \setminus S$ (which we call {\em outgoing edges}) are incident with one of the outgoing vertices of $S$.\end{enumerate}

Then, any Hamiltonian cycle $H$ in $G$ will contain precisely one incoming edge and one outgoing edge, such that if the incoming edge is incident with incoming vertex $i_j$, then the outgoing edge will be incident with outgoing vertex $o_j$. Thus, from the perspective of HCP, the subgraph $S$ functions the same way as a vertex within a larger graph $G$, in that it must be visited precisely once due to the single visit condition. However, although multiple edges may enter and exit $S$, once a particular incoming edge is chosen, the set of possible outgoing edges is reduced to only those incident with the corresponding outgoing vertex, due to the paired vertices condition. Therefore, in-out subgraphs can be used to convert certain constrained forms of HCP into standard HCP, such as can be solved by the excellent solvers due to Baniasadi et al. \cite{slh}, Chalaturnyk \cite{chalaturnyk} or Helsgaun \cite{LKH}.

For some problems, it is often necessary to replace most or all of the vertices in a graph with in-out subgraphs, and hence it is desirable to use in-out subgraphs of small size whenever possible, to prevent the order of the resulting graph from growing too large. In the following, we prove that a $k$-in-out graph must contain at least $2k-1$ vertices. We also prove that any $k$-in-out graph containing $2k-1$ vertices must have at least $4k-4$ edges. We then provide a construction which contains $2k-1$ vertices for all $k \geq 4$. For even $k$ it contains $4k-4$ edges, while for odd $k$ it contains $4k-3$ edges. The cases when $k = 1,2,3$ are handled separately, and contain 1, 3, and 6 vertices respectively; we show that $k = 3$ is the only size of in-out graph where it is impossible to meet the $2k-1$ lower bound on the number of vertices. We also show that our construction provides a planar graph if $k \neq 1 \mod 4$; this may be useful in the case that the original instance was planar and that it is desirable to retain the planarity. We give an example of the usage of in-out graphs by converting the {\em generalized traveling salesman problem} (GTSP) to the standard traveling salesman problem. Most conversions of GTSP to TSP described in literature to date have required the introduction of large weights which are required to grow with the order of the instance, which in turn often leads to numerical problems. The construction we give using in-out subgraphs avoids this issue altogether. Finally, we provide a set of constraints that can be used in cutting-plane approaches whenever an in-out subgraph is used in the context of HCP or TSP.


\section{Bounds}

In this section, we consider bounds on the number of vertices and edges required to induce the in-out property. First, we consider the minimum number of vertices required for an in-out graph.

\begin{proposition}Every $k$-in-out graph $S$ has number of vertices (order) at least $2k-1$.\label{prop-order}\end{proposition}

\begin{proof}Consider an in-out graph $S$ with order $v$. Without loss of generality, it is possible to label the vertices so that incoming vertex $i_1 = 1$, outgoing vertex $o_1 = v$, and there is a path between them traversing vertices $1 \rightarrow 2 \rightarrow 3 \rightarrow \cdots \rightarrow v$.

Now, suppose there is some incoming vertex $i$ and some outgoing vertex $o$ such that, with this labelling, $o$ directly precedes $i$. Such a situation is illustrated in Figure \ref{fig1}. It clear that, in this scenario, it is possible to find a union of two disjoint paths in $S$, the first starting at $i_1$ and finishing at $o$, and the second starting at $i$ and finishing at $o_1=v$, that covers all the vertices of $S$. However, this is impossible because, from Definition \ref{def-ios}, $S$ must satisfy the single visit condition. Hence, it can never be the case that an outgoing vertex is directly followed by an incoming vertex in our chosen labelling.

\begin{figure}[h]\begin{center}\includegraphics[scale=0.425]{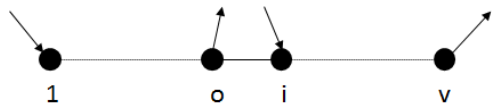}
\caption{The scenario where an outgoing vertex directly precedes an incoming vertex on the Hamiltonian path between a different pair of incoming and outgoing vertices. In such a case the single visit condition is violated, and so this can never be the case for an in-out graph.}\label{fig1}\end{center}\end{figure}

There are $k$ outgoing vertices in $S$. One of them ($o_1$) is labelled $v$, so no vertex follows it in our labelling. However, for each of the other $k-1$ outgoing vertices, there is a vertex which follows it which, as argued above, cannot be an incoming vertex. Hence, $S$ contains $k-1$ vertices which cannot be incoming vertices, plus $k$ incoming vertices, and so must contain at least $2k-1$ vertices.\end{proof}

Suppose that a $k$-in-out graph has $2k-1$ vertices. We next consider the minimum number of edges this graph must contain.

\begin{theorem}Any $k$-in-out graph of order $2k-1$ contains at least $4k-4$ directed edges.\label{thm-edges}\end{theorem}

\begin{proof}Consider a $k$-in-out graph $S$ of order $2k-1$, and suppose that $S$ contains a minimal number of edges. The set of vertices $V$ in $S$ can be partitioned into four disjoint subsets as follows. Denote by $I$ the set of incoming vertices which are not also outgoing vertices. Denote by $O$ the set of outgoing vertices which are not also incoming vertices. Denote by $B$ the set of vertices which are both incoming and outgoing vertices. Denote by $N$ the set of vertices which are neither incoming nor outgoing vertices. We denote their cardinalities as $a = |I| = |O|$, $b = |B|$ and $c = |N|$. Then, since $a+b = k$ and $2a + b +c = 2k-1$, it can be seen that $c = b-1$. Hence the number of vertices in $N$ is one fewer than the number of vertices in $B$.

The proof will be organised as follows. First, we will consider the case where $a = 0$ and show that this case is trivial, and so we will then restrict attention to the case where $a \geq 1$. First we will show that no edges can go from vertices in $O$ to vertices in $I$. Then we will show that, for any labelling of the vertices such that $1 \rightarrow 2 \rightarrow \cdots \rightarrow 2k-1$ is a Hamiltonian path in $S$, the vertices of the graph are divided into segments each separated by a single vertex in $N$, and each containing exactly one vertex from $B$. We will then prove edges emanating from vertices in $I$ and $N$ can only go to vertices in $I$ or $B$, and similarly, edges emanating from vertices in $B$ and $O$ can only go to vertices in $O$ or $N$. Finally, we will use the paired edge condition to determine the minimal number of edges required.

Consider first the case where $a = 0$, that is, all incoming vertices are also outgoing vertices. We note that this is the case for the construction given later in this manuscript. Then $b = k$ and $c = k-1$. Suppose that $S$ is labelled in order of one of the Hamiltonian paths. We will refer to such a labelling as a {\em path-labelling}. Consider any two consecutive vertices $x$ and $y$. It is clear that they cannot both be in $B$, otherwise it would be possible to start at the initial vertex, travel along the path to $x$ and depart, then re-enter at $y$ and complete the rest of the path, violating the single visit condition. Hence, by the pigeonhole principle, the vertices must be ordered starting with a vertex from $B$, then a vertex from $N$, then a vertex from $B$, and so on. Since this must be true for any path-labelling in $S$, it is obvious that any edges between two vertices in $B$ are unnecessary, and similarly, any edges between two vertices of $N$ are unnecessary. Since $S$ is optimal, none of these unnecessary edges exist and hence $S$ is bipartite. Now consider any consecutive vertices $x \in N$ and $y \in B$. Since $y \in B$, it is an incoming vertex, and hence a Hamiltonian path must exist which terminates in its corresponding outgoing vertex. Since $x \in N$, it cannot be the corresponding outgoing vertex. Hence, the Hamiltonian path must eventually reach $x$ and then leave it. Since it cannot return to $y$, there must be at least one more edge emanating from $x$ besides $(x,y)$. Since this must be true for any $x \in N$, we conclude that all vertices in $N$ have out-degree at least 2. Similarly, consider any two consecutive vertices $x \in B$ and $y \in N$. Since $x$ is an outgoing vertex, there must be a Hamiltonian path which ends at $x$, and which cannot have started at $y$. Hence, there must be at least one more edge going into $y$ besides $(x,y)$. Since this is true for any $y \in N$, we conclude that all vertices in $N$ have in-degree at least 2. Since $S$ is bipartite, the set of edges departing vertices in $N$ is disjoint with the set of edges entering vertices in $N$. Hence, at least $4c = 4k-4$ edges are required.

Next, consider the case where $a \geq 1$. It is clear that $b = c + 1 \geq 1$. Suppose that $S$ is labelled with a path-labelling starting from a vertex in $I$. Consider any two consecutive vertices $x$ and $y$. Using an identical argument to the previous paragraph, it is clear that it cannot be the case that $x \in O \cup B$ and $y \in I \cup B$ or else the single visit condition is violated. Then consider any vertex $x \in B$. Clearly the vertex which succeeds it must be from $O \cup N$. If that vertex is from $O$, then the next vertex must also be from $O \cup N$, and so forth. Hence, for any path-labelling of $S$ according to a Hamiltonian path, it must be the case that any two vertices from $B$ which appear consecutively in the path-labelling have at least one vertex from $N$ in between them. Since $b = c + 1$, the pigeonhole principle implies that there will be precisely one vertex from $N$ between them. So the vertices in $B$ and $N$ come in alternating order for any path-labelling. We will say that, given a path-labelling of $S$, the vertices of the graph can be divided into {\em segments} plus the vertices from $N$. That is, the first segment contains the vertices labelled $1, 2, \hdots, j-1$ where $j$ is the first vertex in $N$, then the second segment contains vertices $j+1, j+2, \hdots, m-1$ where $m$ is the second vertex in $N$, and so on. Then each segment contains precisely one vertex from $B$. From the above arguments it is clear that segments must start with some number of (possibly zero) vertices from $I$, then a single vertex from $B$, followed by some number of (possibly zero) vertices from $O$. Since this must be the case for all path-labellings, and $S$ has a minimal number of edges, we can conclude the following:

\begin{itemize}\item Edges which emanate from vertices in $I$ can only go to vertices in $I \cup B$.
\item Edges which emanate from vertices in $O$ can only go to vertices in $O \cup N$.
\item Edges which emanate from vertices in $B$ can only go to vertices in $O \cup N$.
\item Edges which emanate from vertices in $N$ can only go to vertices in $I \cup B$.\end{itemize}

Now, consider any vertex $y \in I \cup B$ in the path-labelling of $S$, except the initial one. It is clear that the vertex $x$ that precedes it will be from $I \cup N$. Then, there must be a Hamiltonian path that begins at $y$ and travels through $x$ at some point, so there must be another edge emanating from $x$ besides $(x,y)$. It can be easily seen that if all vertices $y \in I \cup B$ are considered in this way, then the union of preceding vertices is equal to $I \cup N$. Hence every vertex in $I \cup N$ must have out-degree at least 2 and so these edges contribute at least $2a + 2c$ edges to $S$.

Next, consider any vertex $x \in O \cup B$ in the path-labelling of $S$, except the final one. It is clear that the vertex $y$ that succeeds it will be from $O \cup N$. Then, there must be a Hamiltonian path that ends at $x$, after having travelled through $y$ previously. Hence, there must be another edge going to $y$ in addition to $(x,y)$. Since we considered $a + b - 1 = k-1$ vertices, there are at least $2k - 2$ edges here, and from above, we know that each of these edges must emanate from vertices in $O \cup B$, so they are disjoint with the set of edges considered in the previous paragraph. Hence, there must be at least $2a + 2c + 2k - 2 = 4k - 4$ edges.\end{proof}

\section{Construction}

In this section, we give a construction that produces $k$-in-out graphs of minimal order for any $k \geq 4$. The cases when $k = 1,2,3$ are considered individually. To begin with, we consider a class of bipartite graphs and show that the single visit condition is satisfied by them.

\begin{lemma}Suppose that a graph $G$ is a bipartite graph, that is, its vertex set can be partitioned into $\{V_1, V_2\}$ such that all edges in $G$ are incident with an element from both $V_1$ and $V_2$. Furthermore, suppose that $|V_1| = |V_2|+1$, and that every incoming vertex and every outgoing vertex is contained in $V_1$. Then $G$ satisfies the single visit condition.\label{lem-bip}\end{lemma}

\begin{proof}Since all incoming vertices are in $V_1$, it is clear that whenever $G$ is entered, a vertex in $V_1$ is visited. From here, because $G$ is bipartite, a vertex in $V_2$ is visited next, then a vertex in $V_1$, and so on until $S$ is departed. This departure must also occur at a vertex in $V_1$ since there are no outgoing vertices in $V_2$. It is clear that during this visit, precisely one more vertex from $V_1$ is visited than from $V_2$. Hence, if $G$ is visited $m$ times, there must be $m$ more vertices of $V_1$ visited than those of $V_2$. However, since $|V_1| = |V_2|+1$, it follows that $G$ must be visited precisely once, and so the single visit condition is satisfied.\end{proof}

Hence, from Proposition \ref{prop-order} and Lemma \ref{lem-bip} it is clear that if a bipartite graph of order $2k-1$ satisfying the conditions of Lemma \ref{lem-bip} also satisfies the paired vertices condition, then it is an optimal $k$-in-out graph with respect to the number of vertices. Call such a graph $\mathcal{S}_k$. Recall that we can verify whether or not $\mathcal{S}_k$ satisfies the paired vertices condition by finding all Hamiltonian paths between pairs of incoming and outgoing vertices in $\mathcal{S}_k$.

For small $k$ we can find such graphs explicitly by exhaustive search. For $k = 1$, $\mathcal{S}_1$ contains a single vertex corresponding to both $i_1$ and $o_1$. For $k = 2$, $\mathcal{S}_2$ contains 3 vertices and the directed edges $(1,2), (2,1), (2,3), (3,2)$ with $i_1 = 1$, $i_2 = 3$, $o_1 = 3$, $o_2 = 1$. It can be checked the $\mathcal{S}_1$ and $\mathcal{S}_2$ both meet the conditions of Lemma \ref{lem-bip} and satisfy the paired vertices condition, and in both cases they contain $2k-1$ vertices and $4k-4$ directed edges. We will leave the case where $k = 3$ for the end of this section.

We now provide a procedure for constructing $k$-in-out graphs $\mathcal{S}_k$ of order $2k-1$ for $k \geq 4$.

{\underline{\bf Construction for $\mathcal{S}_k$}}

For even $k \geq 4$, $\mathcal{S}_k$ contains the following edges:

\begin{itemize}\item Undirected edges $(4i-2,4i-1)$, $(4i-1,4i)$ and $(4i, 4i+1)$ for $i = 1, \hdots, \frac{k-2}{2}$.
\item Directed edges $(4i-2,4i+5)$ and $(4i+1,4i+2)$ for $i = 1, \hdots, \frac{k-4}{2}$.
\item Directed edges $(1,2)$, $(2k-6,2k-1)$, $(2k-3,2k-2)$, $(2k-2,1)$, $(2k-2,5)$, $(2k-1,2k-2)$.
\end{itemize}

For odd $k \geq 5$, $\mathcal{S}_k$ contains the following edges:

\begin{itemize}\item Undirected edges $(4i-2,4i-1)$, $(4i-1,4i)$ and $(4i, 4i+1)$ for $i = 1, \hdots, \frac{k-1}{2}$.
\item Directed edges $(4i-2,4i+5)$ and $(4i+1,4i+2)$ for $i = 1, \hdots, \frac{k-3}{2}$.
\item Directed edges $(1,2)$, $(2k-4,1)$, $(2k-2,5)$.
\end{itemize}


In both cases, we define the incoming vertices to be $i_j = 2j-1$ for $j = 1, \hdots, k$. The outgoing vertices require a bit more care to define. In both cases, we can define the majority of the outgoing vertices as $o_{2j} = 4j+3$ and $o_{2j+1} = 4j-3$ for $j = 1, \hdots, \lfloor\frac{k-3}{2}\rfloor$. Then, for the case where $k$ is even, the remaining outgoing vertices yet to be defined are $o_1 = 3$, $o_{k-2} = 2k-1$, $o_{k-1} = 2k-7$, $o_k = 2k-3$. For the case where $k$ is odd, the remaining outgoing vertices yet to be defined are $o_1 = 2k-1$, $o_{k-1} = 3$ and $o_k = 2k-5$. An example of each of the two constructions is displayed in Figure \ref{fig2}.

\begin{figure}[h]\begin{center}\includegraphics[scale=0.333]{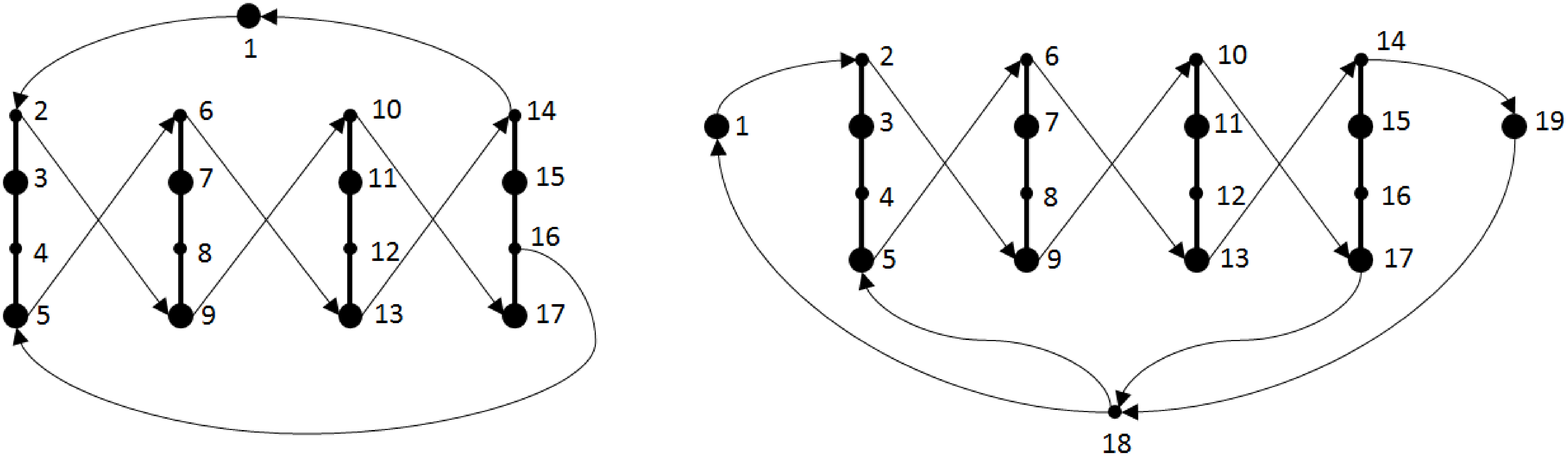}
\caption{In-out graphs constructed from the above construction for $k = 9$ and $k = 10$. The solid thick edges are undirected, and the large vertices are the incoming/outgoing vertices. For larger $k$, the middle pattern is simply repeated as many times as necessary.}\label{fig2}\end{center}\end{figure}

It can be easily checked that $\mathcal{S}_k$ is a bipartite graph satisfying the conditions of Lemma 1. Hence, all that remains is to check the Hamiltonian paths between pairs of incoming and outgoing vertices. In Theorem \ref{thm-paired} we will prove that there are no such Hamiltonian paths between incoming vertices $i_j$ and outgoing vertices $o_m$ for $j \neq m$, and in Proposition \ref{prop-paths} we will show that there is a Hamiltonian path between $i_j$ and $o_j$ for all $j$.

\begin{theorem}
For $k \geq 4$, there are no Hamiltonian paths in $\mathcal{S}_k$ starting at incoming vertex $i_j$ and finishing at any outgoing vertex $o_m$ for $j \neq m$.
\label{thm-paired}\end{theorem}

\begin{proof}
We present the full proof for the even case here. The proof for the odd case follows from analogous arguments. We will partition the incoming vertices into several categories, namely:

\begin{itemize}\item Incoming vertex $i_1$
\item Incoming vertices $i_{2j}$ for $j = 1, 2, \hdots, \frac{k-4}{2}$
\item Incoming vertices $i_{2j+1}$ for $j = 1, 2, \hdots, \frac{k-2}{2}$
\item Incoming vertex $i_{k-2}$
\item Incoming vertex $i_k$
\end{itemize}

Consider first incoming vertex $i_1 = 1$. Consider a Hamiltonian path $P$ that begins at this vertex and finishes at an outgoing vertex. After $P$ begins at vertex 1, it must proceed to vertex 2, and then there is a choice to proceed either to vertex 3 or 9. Suppose that $P$ proceeds to 3, then it is forced to further proceed to 4 and 5. However, at some stage in the future, $P$ must reach vertex $2k-2$, at which point the only options are to either proceed to vertex 1, or vertex 5. Neither choice is still valid, and $P$ cannot conclude here since $2k-2$ is not an outgoing vertex. Hence, we conclude that $P$ must not proceed from 2 to 3. However, since vertex 3 is of degree 2, and if $P$ does not proceed from 2 to 3, then $P$ must finish at vertex $3 = o_1$.

Next consider any incoming vertex $i_{2j} = 4j-1$ for $j = 1, 2, \hdots, \frac{k-4}{2}$. Consider a Hamiltonian path $P$ that begins at this vertex and finishes at an outgoing vertex. After $P$ begins at vertex $4j-1$, it must proceed to the degree 2 vertex $4j$, or else when that vertex is visited later there will be nowhere to go and $P$ would finish here, which is a contradiction since $4j$ is not an outgoing vertex. So $P$ will go to $4j$ and then continue on to vertices $4j+1$ and $4j+2$. At this point, there is a choice to either proceed to vertices $4j+3$ or $4j+9$. Suppose that $P$ proceeds to $4j+3$, then it is forced to further proceed to vertices $4j+4$ and $4j+5$. However, at some stage in the future, $P$ must reach vertex $4j-2$, at which point the only options are to either proceed to vertex $4j-1$ or $4j+5$. Neither choice is still valid, and $P$ cannot conclude here since $4k-2$ is not an outgoing vertex. Hence, we conclude that $P$ must not proceed from $4j+2$ to $4j+3$. However, since vertex $4j+3$ is of degree 2, and $P$ does not proceed from $4j+2$ to $4j+3$, then $P$ must finish at vertex $4j+3 = o_{2j}$.

Next consider any incoming vertex $i_{2j+1} = 4j+1$ for $j = 1, 2, \hdots, \frac{k-2}{2}$. Consider a Hamiltonian path $P$ that begins at this vertex and finishes at an outgoing vertex. After $P$ begins at vertex $4j+1$, it must proceed to the degree 2 vertex $4j$, or else when that vertex is visited later there will be nowhere to go and $P$ would finish here, which is a contradiction since $4j$ is not an outgoing vertex. So $P$ will go to $4j$ and then continue on to vertices $4j-1$ and $4j-2$. However, at some stage in the future, vertex $4j-3$ will be visited. If $j = 1$, this is vertex 1 and it can only go to vertex $2$ which has already been visited, and so the path must finish here. If $j > 1$ then there are two cases to consider. Either vertex $4j-3$ is preceded by $4j-4$, or not. In the latter case, then upon arriving at vertex $4j-3$, vertex $4j-4$ is the only remaining destination (since vertex $4j-2$ has already been visited), which is then followed by $4j-5$ and $4j-6$, at which point there is nowhere to go. Since $4j-6$ is not an outgoing vertex, this case must not have occurred. Hence, vertex $4j-3$ is preceded by $4j-4$ and so upon arriving at vertex $4j-3$ there is nowhere left to go and $P$ must finish at vertex $4j-3 = o_{2j+1}$.


Next consider incoming vertex $i_{k-2} = 2k-5$. Consider a Hamiltonian path $P$ that begins at this vertex and finishes at an outgoing vertex. After $P$ begins at vertex $2k-5$, it must proceed to the degree 2 vertex $2k-4$, or else when that vertex is visited later there will be nowhere to go and $P$ would finish here, which is a contradiction since $2k-4$ is not an outgoing vertex. So $P$ will go to $2k-4$ and then continue on to vertices $2k-3$ and $2k-2$. However, at some stage in the future, vertex $2k-1$ will be visited. This vertex can only go to vertex $2k-2$, which has already been visited. Since there is nowhere left to go, $P$ must finish at vertex $2k-1 = o_{k-2}$.

Finally, consider incoming vertex $i_k = 2k-1$. Consider a Hamiltonian path $P$ that begins at this vertex and finishes at an outgoing vertex. After $P$ begins at vertex $2k-1$ it is forced to visit vertex $2k-2$. At this point, vertex $1$ must be visited as it will not be possible to reach $1$ otherwise. Then for each subsequent vertex, the same argument can be made: upon visiting vertex $j$ we must visit vertex $j+1$ or else it will be impossible to return later. Hence all remaining vertices are visited, with $P$ finishing at vertex $2k-3 = o_k$.

In each case, we have shown that any Hamiltonian path of $\mathcal{S}_k$ which starts at $i_j$ does not end at $o_m$ if $j \neq m$, completing the proof.
\end{proof}

\begin{proposition}For any $k \geq 4$, there is a Hamiltonian path in $\mathcal{S}_k$ between each pair of vertices $i_j$ and $o_j$ for $j = 1, \hdots, k$.\label{prop-paths}\end{proposition}

\begin{proof}It suffices to provide the paths. First, for the case where $k$ is even:

From $i_1$: Starting from 1, go to 2. Then repeat the path $4m+1$ to $4m$ to $4m-1$ to $4m-2$ for $m = 2, \hdots, \frac{k-2}{2}$, followed by $2k-1$ to $2k-2$ to $5$ to $4$ to $3$.

From $i_{2j}$ for $j = 1, 2, \hdots, \frac{k-4}{2}$: Starting from $4j-1$, go to $4j$ to $4j+1$ to $4j+2$. Then repeat the path $4m+1$ to $4m$ to $4m-1$ to $4m-2$ for $m = j+2, j+3, \hdots, \frac{k-2}{2}$, followed by $2k-1$ to $2k-2$ to $1$. Then travel in vertex order along $2, 3, \hdots, 4j-2$. Finally, go to $4j+5$ to $4j+4$ to $4j+3$.

From $i_{2j+1}$ for $j = 1, 2, \hdots, \frac{k-2}{2}$: Starting from $4j+1$, go to $4j$ to $4j-1$ to $4j-2$. Then repeat the path $4m+1$ to $4m$ to $4m-1$ to $4m-2$ for $m = j+1, j+2, \hdots, \frac{k-2}{2}$, followed by $2k-1$ to $2k-2$ to $1$. Then travel in vertex order along $2, 3, \hdots, 4j-3$.

From $i_{k-2}$: Starting from $2k-5$, go to $2k-4$ to $2k-3$ to $2k-2$ to $1$. Then travel in vertex order along $2, 3, \hdots, 2k-6$ and finally go to $2k-1$.

From $i_k$: Starting from $2k-1$, go to $2k-2$ to $1$. Then travel in vertex order along $2, 3, \hdots, 2k-3$.

Next, for the case where $k$ is odd:

From $i_1$: Simply travel in vertex order along $1, 2, \hdots, 2k-1$.

From $i_{2j}$ for $j = 1, 2, \hdots, \frac{k-3}{2}$: Starting from $4j-1$, go to $4j$ to $4j+1$ to $4j+2$. Then repeat the path $4m+1$ to $4m$ to $4m-1$ to $4m-2$ for $m = j+2, j+3, \hdots, \frac{k-1}{2}$, followed by 1. Then travel in vertex order along $2, 3, \hdots, 4j-2$. Finally, go to $4j+5$ to $4j+4$ to $4j+3$.

From $i_{2j+1}$ for $j = 1, 2, \hdots, \frac{k-1}{2}$: Starting from $4j+1$, go to $4j$ to $4j-1$ to $4j-2$. Then repeat the path $4m+1$ to $4m$ to $4m-1$ to $4m-2$ for $m = j+1, j+2, \hdots, \frac{k-1}{2}$, followed by 1. Then travel in vertex order along $2, 3, \hdots, 4j-3$.

From $i_{k-1}$: Starting from $2k-3$, go to $2k-4$ to $1$ to $2$. Then repeat the path $4m+1$ to $4m$ to $4m-1$ to $m-2$ for $m = 2, \hdots, \frac{k-3}{2}$. Finally, go to $2k-1$ to $2k-2$ to $5$ to $4$ to $3$.\end{proof}

\begin{theorem}For any $k \geq 4$, the graph $\mathcal{S}_k$ is a $k$-in-out graph of optimal size.\end{theorem}

\begin{proof}From Lemma \ref{lem-bip} it is clear that $\mathcal{S}_k$ satisfies the single visit condition, while from Theorem \ref{thm-paired} and Proposition \ref{prop-paths} it is clear that $\mathcal{S}_k$ satisfies the paired vertices condition. Hence from Definition \ref{def-ios} it follows that $\mathcal{S}_k$ is a $k$-in-out graph. Then, from Proposition \ref{prop-order} we can see that $\mathcal{S}_k$ is of minimal order, completing the proof.\end{proof}

Next, we consider the number of edges in our construction. For even $k$, $\mathcal{S}_k$ contains $4k-4$ directed edges, and for odd $k$, $\mathcal{S}_k$ contains $4k-3$ directed edges. Recall from Theorem \ref{thm-edges} that any $k$-in-out graph of order $2k-1$ must contain at least $4k-4$ edges. Hence, for even $k$ the construction is also optimal with respect to the number of edges. For odd $k$, the construction is this manuscript does not quite meet the bound provided by Theorem \ref{thm-edges}. However, so far no examples of $k$-in-out graphs with $4k-4$ edges have been found for any odd $k$, which leads to the following conjecture.

\begin{conjecture}Any $k$-in-out graph of order $2k+1$ for odd $k$ has at least $4k-3$ edges.\end{conjecture}

The construction provided in this manuscript produces in-out graphs which are bipartite, and this is also the case in the constructions for $k = 1$ and $k = 2$ given earlier. We now show that, in addition to being bipartite, the in-out graphs we construct are often also planar. This is certainly the case for $k = 1$ and $k = 2$. Consider $\mathcal{S}_k$ defined as above for $k \geq 4$.

\begin{proposition}$\mathcal{S}_k$ is planar unless $k = 1 \mod 4$.\label{prop-planar}\end{proposition}

\begin{proof}We will consider separately the case where $k$ is even, and the case where $k$ is odd. Suppose first $k$ is even. It is clear that, for the embedding displayed in Figure \ref{fig2}, the only edge crossings occur between successive sets of three undirected edges, we call these {\em poles}. This can be avoided by \lq\lq untwisting" every second pole by effectively turning them upside down. The only potential issue is if there is an even number of poles. In this case, vertex $2k-1$ can be relocated underneath the in-out graph to permit a planar embedding, as shown in Figure \ref{fig3}. Hence, $S$ is planar if $k$ is even.

\begin{figure}[h]\begin{center}\includegraphics[scale=0.32]{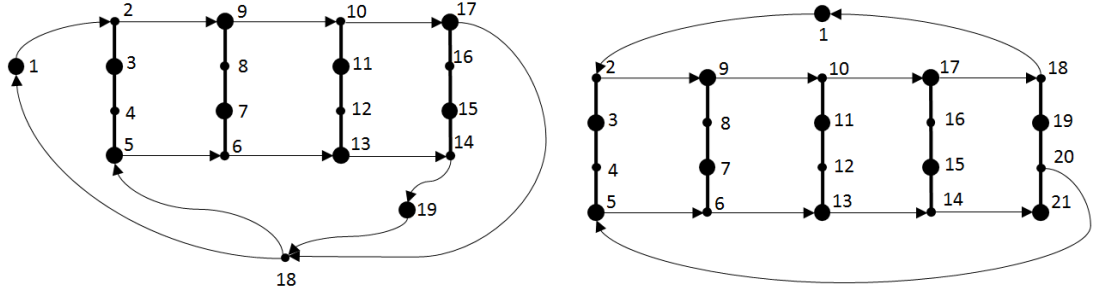}
\caption{Planar embeddings of in-out graphs for $k = 10$ and $k = 11$.}\label{fig3}\end{center}\end{figure}

Next, consider the case where $k$ is odd. Again, for the embedding displayed in Figure \ref{fig2} the only edge crossings occur between successive poles, so we can again untwist every second pole. The only issue occurs when there is an even number of poles. Each pole contains four vertices, two of which are incoming (and outgoing) vertices. Since there is an even number of poles, the number of incoming vertices contained in them is a product of four. Finally, there is one additional incoming vertex, so in this case, $k = 1 \mod 4.$ In this case, $S_k$ has a crossing number of 1; an embedding of $S_9$ with a single edge-crossing is displayed in Figure \ref{fig4}.\end{proof}

\begin{figure}[h]\begin{center}\includegraphics[scale=0.33]{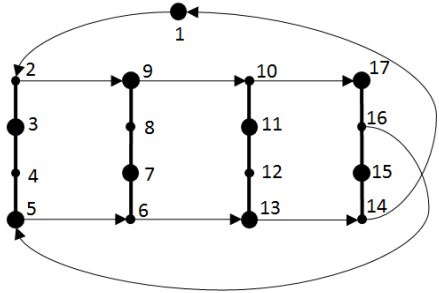}
\caption{An embedding of the in-out graph for $k = 9$ with a single edge-crossing.}\label{fig4}\end{center}\end{figure}

Obviously, if an planar graph is desired for $k = 1 \mod 4$, a $(k+1)$-in-out graph can just be constructed instead with one of the incoming/outgoing vertices treated as a neutral vertex.

Finally, we consider the remaining case when $k = 3$. Exhaustive search shows that there are no 3-in-out graphs on 5 vertices, and hence $k = 3$ is the only case where it is impossible to construct a $k$-in-out graph with $2k-1$ vertices. For $k = 3$, the minimal examples occur for six vertices, and among those, the fewest number of directed edges possible is ten. An example of one such 3-in-out graph, $\mathcal{S}_3$, can be constructed by taking the (directed) path graph $P_6$ on six vertices and adding the directed edges $(1,5)$, $(2,1)$, $(3,2)$, $(5,1)$ and $(6,4)$. Then $i_1 = 1$, $i_2 = 3$ and $i_3 = 6$, while $o_1 = 6$, $o_2 = 4$ and $o_3 = 3$. The three Hamiltonian paths between pairs of incoming and outgoing vertices are $P_1 = 1 \rightarrow 2 \rightarrow 3 \rightarrow 4 \rightarrow 5 \rightarrow 6$, $P_2 = 3 \rightarrow 2 \rightarrow 1 \rightarrow 5 \rightarrow 6 \rightarrow 4$ and $P_3 = 6 \rightarrow 4 \rightarrow 5 \rightarrow 1 \rightarrow 2 \rightarrow 3$. This 3-in-out graph is displayed in Figure \ref{fig5}. Note that the resulting graph is planar, although it is not bipartite; indeed, no 3-in-out graphs on 6 vertices are bipartite. The smallest bipartite 3-in-out graph can be obtained by taking $\mathcal{S}_4$ and then simply treating one of the incoming/outgoing vertices as a neutral vertex.

\begin{figure}[h]\begin{center}\includegraphics[scale=0.33]{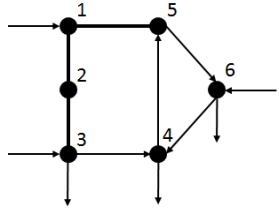}
\caption{A minimal 3-in-out graph $\mathcal{S}_3$.}\label{fig5}\end{center}\end{figure}

%
%
%
%
%

\section{Generalized Traveling Salesman Problem}

We now consider the generalized traveling salesman problem (GTSP) and show that we can convert it to an instance of {\em asymmetric TSP} (ATSP) through the use of in-out subgraphs. First, we recall the definition of the traveling salesman problem (TSP). Consider any graph $G$ and a set of weights $w_{ij}$ on every directed edge $(i,j) \in G$. Any path in $G$ has {\em path length} equal to the sums of weights of the edges used in the path. Then TSP can be defined as the problem of identifying the Hamiltonian cycle of $G$ with shortest path length. If the weights are different depending on the direction in which the edge is traversed, or if some edges can only be traversed in one direction, the problem is often called asymmetric TSP.

A specialisation of TSP is GTSP, wherein the vertices $V$ of $G$ are partitioned into disjoint groups $V_i$, such that $V$ is the union of all $V_i$. Then, GTSP is equivalent to ATSP, except the requirement to visit each vertex from $V$ is replaced by the requirement to visit exactly one vertex from each $V_i$. This variation of ATSP has been considered in various contexts, including order picking in warehouses \cite{daniels}, routing of clients through welfare agencies \cite{saskena}, and computer file sequencing \cite{henry}.

Rather than develop a specialised algorithm for solving GTSP, a common approach in literature has been to convert instances of GTSP into instances of ATSP, so as to take advantage of the wealth of excellent open-source TSP solvers available such as Concorde \cite{concorde} or LKH \cite{LKH}. To the best of the authors' knowledge, the earliest such conversion is due to Lien et al. \cite{lien}, which involved replacing each group with a special subgraph that ensured each group would be visited exactly once. For an instance of GTSP with $n$ vertices and $g$ groups, the conversion by Lien et al. results in an instance of ATSP with $3n + g + 2$ vertices.

Later in the same year, Noon and Bean \cite{noonbean} advocated an alternative approach of adding edges between the vertices in each $V_i$, so as to introduce a cycle with zero weight for each $V_i$, and then adding a large weight to all edges going between vertices in different groups. Since the large weight renders these edges undesirable, a TSP solver will seek to use as few of them as possible and hence will visit each group only once. This conversion results in an instance of ATSP with only $n$ vertices (that is, there is no growth in the order of the instance) but the price paid is the introduction of large weights on $O(n^2)$ edges. Each of these weights must be at least as large as the sum of the $n$ largest weights in the original instance, so the magnitude of the weights grows with the size of the instance. Noon and Bean point out in their manuscript that while the large weights pose no theoretical issues, they create practical difficulties with solving. They indicate that methods such as subtour elimination algorithms will require many branches before the the first non-zero bounds (ie bounds which include any of the edges of large weight) are reached. Instead, they advocate cutting-plane approaches, but indicate that the large weights will pose numerical stability problems for any LP solvers and would inhibit variable elimination.

A few years later, Dimitrijevi\'{c} and \v{S}ari\'{c} \cite{dimitrijevic} proposed a conversion which results in $2n$ vertices, but only introduces $n$ large weights. Then, a few years after that, Behzad and Modarres \cite{behzad} developed another conversion of GTSP to ATSP which, from an algorithmic perspective, performs equivalently to the conversion by Noon and Bean and hence contains $n^2$ large weights as well.

Here, we propose a new alternative. Through the use of in-out subgraphs, we can convert GTSP to ATSP, where the single visit condition will ensure the groups are only visited once, and the paired vertices condition can be used to ensure the appropriate weights are given to each outgoing edge. Since the structure of in-out graphs allows us to satisfy the requirements of GTSP, we are able to avoid the need to use large weights. The following procedure will construct an instance of ATSP from any given instance of GTSP.

\begin{enumerate}\item For each group $V_i$ in the original instance, the new instance should contain a $k$-in-out subgraph $\mathcal{S}_k^i$, where $k = |V_i|$. The weight on each of the edges of $\mathcal{S}_k^i$ should be 0.
\item For every directed edge $(u,v)$ with weight $w_{uv}$ in the original instance, do the following. If $u$ is the $s$-th vertex in $V_i$ and $v$ is the $r$-th vertex in $V_j$, then add an edge to the new instance between outgoing vertex $o_s$ of $\mathcal{S}_k^i$ and incoming vertex $i_r$ of $\mathcal{S}_k^j$ with weight $w_{uv}$.
\end{enumerate}

If the original instance has $n$ vertices, partitioned into $g$ groups, and with $m$ groups having cardinality 3, the order of the resultant instance will be $2n - g + m$. This is obviously superior to the conversion due to Lien et al. It is also superior to the conversion by Dimitrijevi\'{c} both in terms of size and also by avoiding the introduction of large weights. Since $g < n$ for any meaningful instance of GTSP, the conversion given here results in a larger instance than that from Noon and Bean. However, by avoiding introducing large weights, our conversion is considerably more numerically stable. We can also partially alleviate the burden of the larger size by taking advantage of constraints for each in-out subgraph as described in the next section.

\section{Constraints for in-out subgraphs}

The use of in-out subgraphs allows us to pose constrained forms of HCP or TSP as standard forms. However, a solver for one of these problems may not be \lq\lq aware" that the in-out subgraph satisfies the single visit condition or the paired vertices condition, and could waste time trying to eliminate possibilities which are already prevented by the in-out subgraphs. Hence, whenever possible, such as for cutting-plane approaches, it is beneficial to include constraints which instruct the solver about the in-out subgraphs. We conclude this manuscript with some such constraints.

In each of the following constraints, it is assumed that $S$ is any $k$-in-out subgraph with incoming vertices $i_j$ and outgoing vertices $o_j$ for $j = 1, 2, \hdots, k$, and that $x_{ij}$ is the variable corresponding to using (directed) edge $(i,j)$ in the tour.

\begin{eqnarray}
\sum_{v \not\in S}\sum_{j=1}^k  x_{v,i_j} & = & 1,\label{eq-in}\\
\sum_{v \not\in S}\sum_{j=1}^k  x_{o_j,v} & = & 1,\label{eq-out}\\
\sum_{v \not\in S} \left( x_{v,i_j} - x_{o_j,v} \right) & = & 0, \quad \forall j = 1, \hdots, k.\label{eq-pair}
\end{eqnarray}

Constraint (\ref{eq-in}) ensures exactly one incoming edge is used, and constraint (\ref{eq-out}) ensures exactly one outgoing edge is used. Constraints (\ref{eq-pair}) ensure that an incoming edge incident with incoming vertex $i_j$ is used if and only if an outgoing edge incident with outgoing vertex $o_j$ is also used.

We can add further constraints if we consider the paths between pairs of incoming and outgoing vertices. For the construction $\mathcal{S}_k$ given in this manuscript, there is a unique Hamiltonian path $P_j$ in $\mathcal{S}_k$ between incoming vertex $i_j$ and outgoing vertex $o_j$ for each $j = 1, 2, \hdots, k$. Hence, we can also add the following constraints, using the shorthand that $x_e = x_{i,j}$ if $e = (i,j)$:

\begin{eqnarray}
(2k-2 + \delta_{3k})\sum_{v \not\in \mathcal{S}_k}x_{v,i_j} - \sum_{e \in P_j} x_e & \leq & 0, \quad \forall j = 1, \hdots, k,\label{eq-rightedges}\\
x_e - \sum_{v \not\in \mathcal{S}_k}\sum_{j | e \in P_j} x_{v,i_j} & = & 0, \quad \forall e \in \mathcal{S}_k.\label{eq-appropriateedges}
\end{eqnarray}

Constraints (\ref{eq-rightedges}) ensure that if an incoming edge incident with incoming vertex $i_j$ is used, then every one of the edges in path $P_j$ must also be visited. Note that $\delta_{3k}$ is the Kronecker delta that is equal to one if $k = 3$ and is zero otherwise; this term is necessary because the Hamiltonian path in $S_3$ contains 5 edges rather than the normal $2k-2$ edges for all other $S_k$. Constraints (\ref{eq-appropriateedges}) ensure that an edge $e$ in $\mathcal{S}_k$ is used if and only if an incoming edge incident with an incoming vertex $i_j$ is used such that $P_j$ contains $e$.

\bibliographystyle{plain}   

\end{document}